\documentclass[12pt]{amsart}
\usepackage[headings]{fullpage}
\usepackage{geometry}       
\usepackage{titlesec}      
\usepackage{url}   

\usepackage[utf8]{inputenc}
\usepackage[english]{babel}
\usepackage{csquotes}
\usepackage[style=numeric,backend=biber,sorting=nyt]{biblatex} 
\DeclareLanguageMapping{english}{english-apa}
\DeclareFieldFormat[article, inbook,incollection,inproceedings,patent,thesis,unpublished]{title}{#1}
\renewbibmacro*{author}{%
  \ifboolexpr{
    test \ifuseauthor
    and
    not test {\ifnameundef{author}}
  }
    {\printnames{author}%
     \setunit{\printdelim{authortypedelim}}
     \printtext[parens]{\printfield{year}}
     \newunit\newblock}
    {}}
\renewbibmacro*{date}{\iffieldundef{year}{}{\clearfield{year}}} 
\renewbibmacro*{issue+date}{%
  \iffieldundef{issue}{}{\printfield{issue}\addspace}%
  \iffieldundef{year}{}{\clearfield{year}}} 

\titleformat{\section}         
  {\bfseries}                  
  {\thesection}                
  {1em}                        
  {}                           
\titleformat{\subsection}      
  {\bfseries\itshape}          
  {\thesubsection}
  {1em}
  {}

\titleformat{\subsubsection}   
  {\itshape}                   
  {\thesubsubsection}
  {1em}
  {}

\usepackage{amsmath,amssymb,amsthm}

\usepackage{graphicx}
\usepackage{enumerate}
\usepackage{xcolor}
\usepackage{url}
\usepackage{slashed}
\usepackage{bm}
\usepackage{tikz-cd}
\usepackage{pb-diagram}

\usepackage{enumitem}
\usepackage[bookmarks=true,%
    colorlinks=true,%
    linkcolor=blue,%
    citecolor=blue,%
    filecolor=blue,%
    menucolor=blue,%
    urlcolor=blue,%
    breaklinks=true]{hyperref}

\usepackage{booktabs}  
\usepackage{array}      
\usepackage{makecell}   

\newtheorem{thm}{Theorem}[]
\newtheorem*{thm*}{Theorem}
\newtheorem{lem}[thm]{Lemma}

\newtheorem{ques}[thm]{Question}

\newtheorem{ex}[thm]{Example}
\newtheorem{rem}[thm]{Remark}

\newcommand{\param}{{\mathchoice{\mkern1mu\mbox{\raise2.2pt\hbox{$
\centerdot$}}
\mkern1mu}{\mkern1mu\mbox{\raise2.2pt\hbox{$\centerdot$}}\mkern1mu}{
\mkern1.5mu\centerdot\mkern1.5mu}{\mkern1.5mu\centerdot\mkern1.5mu}}}

\renewcommand{\setminus}{{\smallsetminus}}

\begin{document}
\title{Small genus, small index critical points of the systole function.}
\author{Ni An}
\author{Ferdinand Ihringer}
\author{Ingrid Irmer}



\begin{abstract}
In this paper the index of a family of critical points of the systole function on Teichm\"uller space is calculated. The members of this family are interesting in that their existence implies the existence of strata in the Thurston spine for which the systoles do not determine a basis for the homology of the surface. Previously, index calculations of critical points with this pathological feature were impossible, because the only known examples were in surfaces with huge genus.

A related concept is that of a ``minimal filling subset'' of the systoles at the critical point. Such minimal filling sets are studied, as they relate to the dimension of the Thurston spine near the critical point. We find an example of a minimal filling set of simple closed geodesics in genus 5 with cardinality 8, that are presumably realised as systoles.
More generally, we determine the smallest and largest cardinality of a minimal filling set related to a tessellation
of a hyperbolic surface by regular, right-angled $m$-gons for $m \in \{ 5, 6, 7 \}$.
For this, we use integer linear programming together with a hand-tailored symmetry breaking technique.
\end{abstract}

\maketitle

\section{Introduction}

The purpose of this paper is to study a family obtained from covering spaces of the examples studied in Example 36 of \cite{SchmutzMorse}. The members of this family, like those from \cite{SchmutzMorse}, are critical points of any mapping class group-equivariant Morse function on the Teichm\"uller space $\mathcal{T}_{g}$ of closed, compact, connected genus $g$ surfaces with $g\geq 2$. Calculating the index of these critical points is a problem related to the topology of the moduli space of closed surfaces, as will now be outlined.

The systoles on a hyperbolic surface $S_{g}$ of genus $g$ are the shortest closed geodesics. Every hyperbolic surface has at least one systole. The systole function $f_{\mathrm{sys}}:\mathcal{T}_{g}\rightarrow \mathbb{R}_{+}$ is the function that maps $x\in \mathcal{T}_{g}$ to the length of the systoles at $x$. It was shown in \cite{Akrout} that $f_{\mathrm{sys}}$ is a topological Morse function. Topological Morse functions are defined in \cite{Morse}, and can be used in the category of topological manifolds in much the same way as Morse functions on differentiable manifolds; a reference for this is Section 3 of Essay III of \cite{KS}. Many examples of critical points of $f_{\mathrm{sys}}$ are now known, for example \cite{CR,SchmutzMaxima,SchmutzExtremal}. The most famous examples are the Bolza surface in genus 2, \cite{Bolza}, and the Klein Quartic, \cite{Klein}. An estimate of the number of critical points of the systole function on $\mathcal{T}_{g}$ was given in \cite{YG}.

The Thurston spine, $\mathcal{P}_{g}$, is a CW complex embedded in the Teichm\"uller space, $\mathcal{T}_{g}$, of a genus $g$ surface, with $g\geq 2$. It consists of all the points at which the systoles cut the surface into polygons, i.e. the systoles \textit{fill} the surface. As explained in \cite{Thurston}, it follows from a theorem first proven by Bers \cite{Bers}, that $\mathcal{P}_{g}$ contains all the critical points of $f_{\mathrm{sys}}$. Moreover, $\mathcal{P}_{g}$ is the image of a mapping class group-equivariant deformation retraction of $\mathcal{T}_{g}$, \cite{Me2023,Thurston}. In \cite{SchmutzMorse}, Schmutz Schaller posed the question of whether there exist critical points of $f_{\mathrm{sys}}$ of co-index larger than the virtual cohomological dimension $4g-5$ of the mapping class group. This question is tied up with the question of whether the dimension of $\mathcal{P}_{g}$ is larger than the lower bound given by $4g-5$. As explained in Section 5 of \cite{MorseSmale}, it is a consequence of the work of Akrout that a critical point $p$ of $f_{\mathrm{sys}}$ of index $j$ determines a cell of $\mathcal{P}_{g}$ of codimension $j$, and that the cardinality of a minimal filling set contained in the set of systoles at $p$ gives an upper bound on the dimension of $\mathcal{P}_{g}$ on a neighbourhood of $p$.
\begin{figure}[!ht]
\centering
\label{resultstable}
\end{figure}

\begin{table}[htbp]
\centering
\begin{tabular}{@{}ccccccc@{}}
\toprule
$ m $ & genus & \multicolumn{1}{c}{number of systoles} & \multicolumn{1}{c}{smallest cardinality} & \multicolumn{1}{c}{largest cardinality} & \multicolumn{1}{c}{rank} & \multicolumn{1}{c}{index} \\
\midrule
5 & 5 & 20 & 8 & 10 & 10 & 12 \\
6 & 17 & 48 & 24 & 31 & 31 & 39 \\
7 & 49 & 112 & 64 & 80 & 84 & 102 \\
8 & 129 & 256 &  $\geq155$, $\leq160$ & $\geq190$, $\leq192$ & 210 & 245 \\
\bottomrule
\end{tabular}

\caption{The results of our computations.  Here ``smallest/largest cardinality'' refers to the smallest/largest cardinality of a minimal filling set. The ``rank'' is the dimension of the homological span of the systoles in $H_{1}(S_{g};\mathbb{Q})$.}

\label{tab:filling_sets}
\end{table}

While it is now known from a number of sources, \cite{Maxime,Coxeter}, that the answer to Schmutz Schaller's question is yes, the examples in this paper are the first such examples for which the index could actually be calculated. The results are given in Table \ref{resultstable} \footnote{Computational data files associated with Table \ref{resultstable} are available on the arXiv version of this article.}. Enough examples were calculated to guess the formula $m2^{m-3}-(m+3)$ for the index. The parameter $m=5,6,7, \ldots$ will be explained in Section \ref{secexamples}.

The first in our family of examples has genus 5. This is likely the smallest genus surface for which the dimension of the Thurston spine is larger than the virtual cohomological dimension of the mapping class group and for which there is a filling set of systoles that do not span $H_{1}(S_{g};\mathbb{Q})$. It is known that neither of these things happens in genus 2, \cite{Calculation} and it appears unlikely in genus 3 as a result of preliminary results due to Schmutz Schaller, \cite{SchmutzMaxima}. The second example in the family was mentioned in \cite{estimating} and has genus 17. Previously, only comparatively large genus examples of this phenomenon were known, making it difficult to study and obtain intuition.

Following Thurston, a \textit{stratum} labelled by $C$ is the set of points in $\mathcal{T}_{g}$ at which $C$ is the set of systoles.

A minimal filling set $C'$ of closed geodesics on $S_{g}$ is a filling set that does not contain any proper subsets that fill.  It was shown in \cite{MorseSmale} that for a minimal filling set $C'$, the codimension in $\mathcal{T}_{g}$ of a stratum of $\mathcal{P}_{g}$ on which the set of systoles is $C'$ is given by $|C'|-1$. Our examples have an unusual property that the set $C$ of systoles contains minimal filling subsets of different cardinalities. We computed enough examples of minimal filling subsets of smallest cardinality to guess the formula $(m-3)2^{m-3}$, as well as studying orbits of minimal filling subsets.

Another property satisfied by all of our minimal filling sets is that the number of complementary regions is equal to one. The cardinality of such filling sets was bounded from above by $2g$ in \cite{Fillingsystems}.

There is a property of systoles known as \textit{local finiteness}, explained in Section \ref{secexamples}. As a consequence of local finiteness, one expects that there is a minimal filling set $C'\subset C$ such that $C'$ is the set of systoles on a locally top-dimensional stratum of $\mathcal{P}_{g}$ adjacent to $p$. All known locally top-dimensional strata of $\mathcal{P}_{g}$ have systoles that are minimal filling sets. For the representatives in our family of examples with even parameter $m$, the result claimed in Theorem 2 of \cite{estimating} ensures that all locally top-dimensional strata adjacent to $p$ have systoles exactly equal to a minimal filling set of $C$ of smallest cardinality. The set $C$ of systoles in our example of a critical point in genus 5 has a minimal filling set of cardinality 8, shown in Figure \ref{8systoles}. This strongly suggests that in genus 5, the dimension of the Thurston spine is larger than the virtual cohomological dimension of the mapping class group.

Lemma 9 of \cite{SchmutzMorse} implies that the index of a critical point is always bounded from below by one less than the largest cardinality of a minimal filling subset of the set of systoles. In the family of examples from Example 36 of \cite{SchmutzMorse}, the index of the critical point is equal to this lower bound. In our examples, all minimal filling sets of systoles have cardinality significantly less than the index.

It was shown in \cite{MorseSmale} that $\mathcal{P}_{g}$ contains a set of unstable manifolds of the critical points of $f_{\mathrm{sys}}$\footnote{As the systole function is only a topological Morse function, these are not necessarily uniquely defined as for smooth Morse functions}. In genus 2, $\mathcal{P}_{g}$ is exactly a set of unstable manifolds of the critical points of $f_{\mathrm{sys}}$. Calculations on our first few elements in the family of examples show that for these critical points, $\mathcal{P}_{g}$ strictly contains the unstable manifolds.



\textbf{Outline of the paper.} Section \ref{secexamples} constructs the examples, derives some of their basic properties and explains the combinatorial characterisation of the systoles. Section \ref{secindex} explains the algorithms for finding minimal filling sets and computing the index. Related open questions are discussed in Section \ref{future}.


\section{The Examples}
\label{secexamples}
The purpose of this section is to introduce the examples studied, their basic properties and some of the concepts from the literature that will be used.

The examples arise from tessellations of a closed, connected, hyperbolic surface $S_{g}$ by right angled, regular $m$-gons ($m=5, 6, 7,\ldots$), with dual graph given by the 1-skeleton $C^{m}$ of the $m$-dimensional cube. It follows from Euler characteristic arguments that the genus of the surface is given by 
\begin{equation}
g=1+(m-4)2^{m-3}
\end{equation}

Denote by $C$ the set of closed geodesics on $S_{g}$ contained in the edges of the right-angled $m$-gons. These will be shown to be the systoles. From the parameterisation of the geodesics in $C$ given in Subsection \ref{subsecparameterisation}, it follows that each geodesic in $C$ traverses exactly 4 edges of the tessellation, and intersects exactly 4 other geodesics in $C$.

The length of a geodesic $c\in C$ is a function $L(c):\mathcal{T}_{g}\rightarrow \mathbb{R}_{+}$. A \textit{length function} generalises the notion of the length of a closed geodesic; it is a function $\mathcal{T}_{g}\rightarrow \mathbb{R}_{+}$ that can be written as a finite positive linear combination of lengths of closed geodesics.

The automorphism group of the surface with this tessellation can be deduced from the isomorphism group of the $m$-dimensional cube, as explained in Subsection \ref{subsecparameterisation}. It acts transitively on the set $C$. The automorphism group is also isomorphic to a quotient of a triangle group $\Delta(\frac{1}{2m}, \frac{1}{4}, \frac{1}{2})$ ; each $m$-gon is made up of $2m$ triangles with angles $\frac{\pi}{2m}$, $\frac{\pi}{4}$, $\frac{\pi}{2}$ as shown in Figure \ref{ngontriangulation}. Examples with automorphism groups given by quotients of triangle groups have been intensely studied, for example \cite{Modular,Number}. This is partly because as shown in Theorem 37 of \cite{SchmutzMorse}, the large symmetry groups force such examples to be critical points of all mapping class group-equivariant Morse functions on $\mathcal{T}_{g}$.

\begin{figure}[!ht]
\centering
\includegraphics[width=0.3\textwidth]{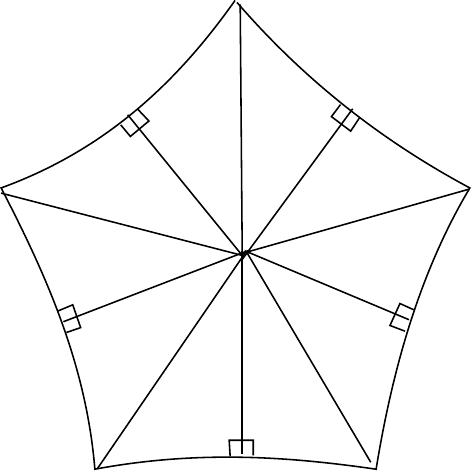}
\caption{Each $m$-gon is triangulated by $2m$ triangles, as shown here in the case $m=5$.}
\label{ngontriangulation}
\end{figure}

Throughout this paper, geodesics are assumed to be simple and closed, and a capital $C$ will be used to denote a set of closed, simple geodesics intersecting pairwise at most once. These are necessary properties for a set of geodesics to be realisable somewhere on $\mathcal{T}_{g}$ as the set of systoles.

Another important property of sets of systoles is \textit{local finiteness}. For given $\epsilon>0$, there is a finite number of geodesics of length less than $f_{\mathrm{sys}}(x)+\epsilon$, \cite{M}, so there exists a nonzero gap between the length of the systoles and the length of the next shortest geodesics; call this gap $g_{sys}(x)$. Suppose $x\in \mathcal{T}_{g}$ is a point at which the systoles are given by the set $C$. It is therefore necessary to deform the marked hyperbolic metric by a nontrivial amount in order to change the length of a curve by at least the amount $g_{sys}(x)$ needed to reach a point $y\in \mathcal{T}_{g}$ at which there is a systole not in the set $C$. In other words, there is a neighbourhood of $x$ in which the systoles are contained in the set $C$.

The number of systoles at any point in $\mathcal{T}_{g}$ is bounded from above by a constant depending on $g$ known as the \textit{kissing number}. This is studied for example in \cite{Kiss} and \cite{SchmutzExtremal}.

\begin{thm}
\label{systolesandtesselation}
In our family of examples, the systoles consist of the set $C$ of closed geodesics contained in the edges of the right-angled $m$-gons.
\end{thm}
\begin{proof}
This is a generalisation of the argument given in the proof of Theorem 36 of \cite{SchmutzMorse}. Let $l_{s}$ be the side length of an $m$-gon. The lengths of the geodesics in $C$ all have length $4l_{s}$. The dual graphs $C^{m}$, $m=5, 6, \ldots$ in the construction of the examples have the property that the shortest closed loops all have length 4. The theorem is proven by assuming the existence of a geodesic $c'\notin C$ on the surface with length at most $4l_{s}$ and deriving a contradiction.

Denote by $A:=\{a_{1}, \ldots, a_{k}\}$ the set of arcs formed by intersecting the geodesic $c'$ with the $m$-gons of the tessellation. The index on the arcs in $A$ is assumed to correspond to the cyclic ordering of the arcs along $c'$.

Arcs in $A$ with endpoints on neighbouring edges of the $m$-gon in which they are contained will be called \textit{corner arcs}. All other arcs in $A$ will be called non-corner arcs. When a non-corner arc $a$ has endpoints on a pair of edges separated by a single edge of the $m$-gon, the length of $a$ is more than $l_{s}$ by Theorem 3.5.7 of \cite{Ratcliffe}. When the endpoints of a non-corner arc are on a pair of edges separated by more than one edge of the $m$-gon, multiple applications of Theorem 3.5.7 of \cite{Ratcliffe} show that the length of $a$ is greater than $l_{s}$.

The geodesic $c'$ is homotopic to a closed loop in the embedding of the graph $C^{m}$ in $S_{g}$. One choice of homotopy takes every arc $a_{i}$ in $A$ to the union of 2-half edges of the embedding of $C^{m}$ in $S_{g}$ with endpoints on the same edges of the $m$-gon containing $a_{i}$. If $c'$ passes through a vertex $v$ of one of the $m$-gons, a choice is made, and a corner arc is added to the set $A$. The existence of a corner arc in $A$ implies the existence of a $4m-8$-gon $G_{i}^{4m-8}$ obtained by gluing the four $m$-gons incident on $v$ along common edges. The polygon $G_{i}^{4m-8}$ has the property that a connected component of $c'\cap G_{i}^{4m-8}$ traverses at least three arcs in $A$, two possible examples of which are shown in red as $a_{j}$ and $a_{j-1}$ in Figure \ref{cornerarcs}. For each corner arc $a$ of $A$, add an edge to $C^{m}$ connecting the pair of diagonally opposite vertices of the square in $C^{m}$ in the center of the $m$-gons containing endpoints of $a$, as shown by the dotted line in Figure \ref{cornerarcs}. The graph $C^{m}$ with a diagonal edge added for each corner arc of $A$ will be called $C^{m}_{d}$.

The number $k$ of arcs in $A$ satisfies $k\geq 4$. This is because the shortest closed loop in $C^{m}$ is 4, so if $k<4$, $c'$ would be contained in a contractible subsurface, and therefore homotopically trivial. For a closed geodesic $c$ on $S_{g}$, the combinatorial length of $c$ is defined to be the smallest number of edges of  $C^{m}_{d}$ traversed by a curve homotopic to $c$. Adding the diagonal edges to $C^{m}$ to obtain $C^{m}_{d}$ does not change the minimum combinatorial length of a homotopically nontrivial curve, which remains 4.

If $A$ contained at least 4 non-corner arcs, the length of $c'$ would be greater than that of the systoles. It follows that if $c'$ is a systole, $A$ must contain corner arcs.

\begin{figure}[!ht]
\centering
\includegraphics[width=0.4\textwidth]{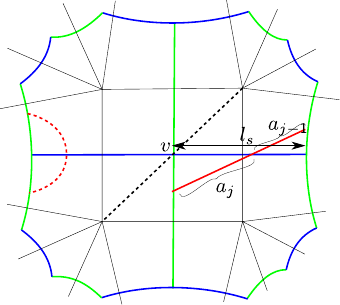}
\caption{Four right angled hexagons (not drawn to scale, with alternating purple and blue edges) of side length $l_{s}$ in the tessellation. The dual $C^{6}$ is shown in grey, and a couple of corner arcs $a_{j}$ and $a_{j-1}$ are shown in red.}
\label{cornerarcs}
\end{figure}

The existence of corner arcs forces $k>4$, because the combinatorial length of $c$ must be at least 4. As shown in Figure \ref{cornerarcs}, if a corner arc $a_{j}$ is preceded by another corner arc $a_{j-1}$ the length of $a_{j-1}$ and $a_{j}$ together is more than $l_{s}$. Note that geodesics are in minimal position, so a pair of consecutive corner arcs cannot make up one side of a bigon, as shown by the dotted red line in Figure \ref{cornerarcs}. Similarly if $a_{j+1}$ is a corner arc, the length of $a_{j}$ and $a_{j+1}$ is more than $l_{s}$. If the arcs on both sides of $a_{j}$ are non-corner arcs, then the lengths of $a_{j-1}$ and $a_{j}$ are greater than $l_{s}$. Since the complement of this pair of arcs in $A$ must have at least cardinality 3 in order to give a curve with combinatorial length at least 4, this gives the required contradiction to the existence of $c'$.
\end{proof}




\subsection{Parameterising the systoles}
\label{subsecparameterisation}
This subsection explains the combinatorial description of the systoles in our family of examples.

Vertices of $C^{m}$ are represented by $m$-tuples with entries in $\mathbb{Z}_{2}$. Edges of $C^{m}$ are determined by the pair of vertices that make up their endpoints. Edges connect vertices given by tuples with all entries except one equal. A square in the $m$-dimensional cube is represented by 4 vertices. These 4 vertices have the property that all entries except 2 are equal. For example, there is a square in $C^{6}$ given by the vertices
 \begin{equation*}
(0,0,0,0,0,0), \ 
(0,0,0,1,0,0), \ 
(0,0,1,0,0,0), \ 
(0,0,1,1,0,0)
\end{equation*}

Note that all but two of the coordinates of the vertices (in the example above, the third and fourth) of a square are the same. The square given above will be represented by $\{(3,4), (0,0,0,0)\}$. More generally, a square in the $m$-dimensional cube will be denoted by a set consisting of a 2-tuple $(i,j)$ with $i\mod m< j\mod m$, $i,j\in \{1, \ldots, m\}$ and an $m-2$-tuple $(a_{1}, \ldots, a_{m-2})$ with entries in $\mathbb{Z}_{2}$.

As shown in Figure \ref{whichsquares}, some of the squares have a boundary that is mapped to a homotopically nontrivial loop on the surface, while others contain a vertex at which four $m$-gons of the tessellation come together and are realised as squares on the surface. A square that is realised as a square on the tessellated surface will be called a \textit{realised square}. The next lemma gives a combinatorial recipe for distinguishing the two types of squares.

\begin{figure}[!ht]
\centering
\includegraphics[width=0.6\textwidth]{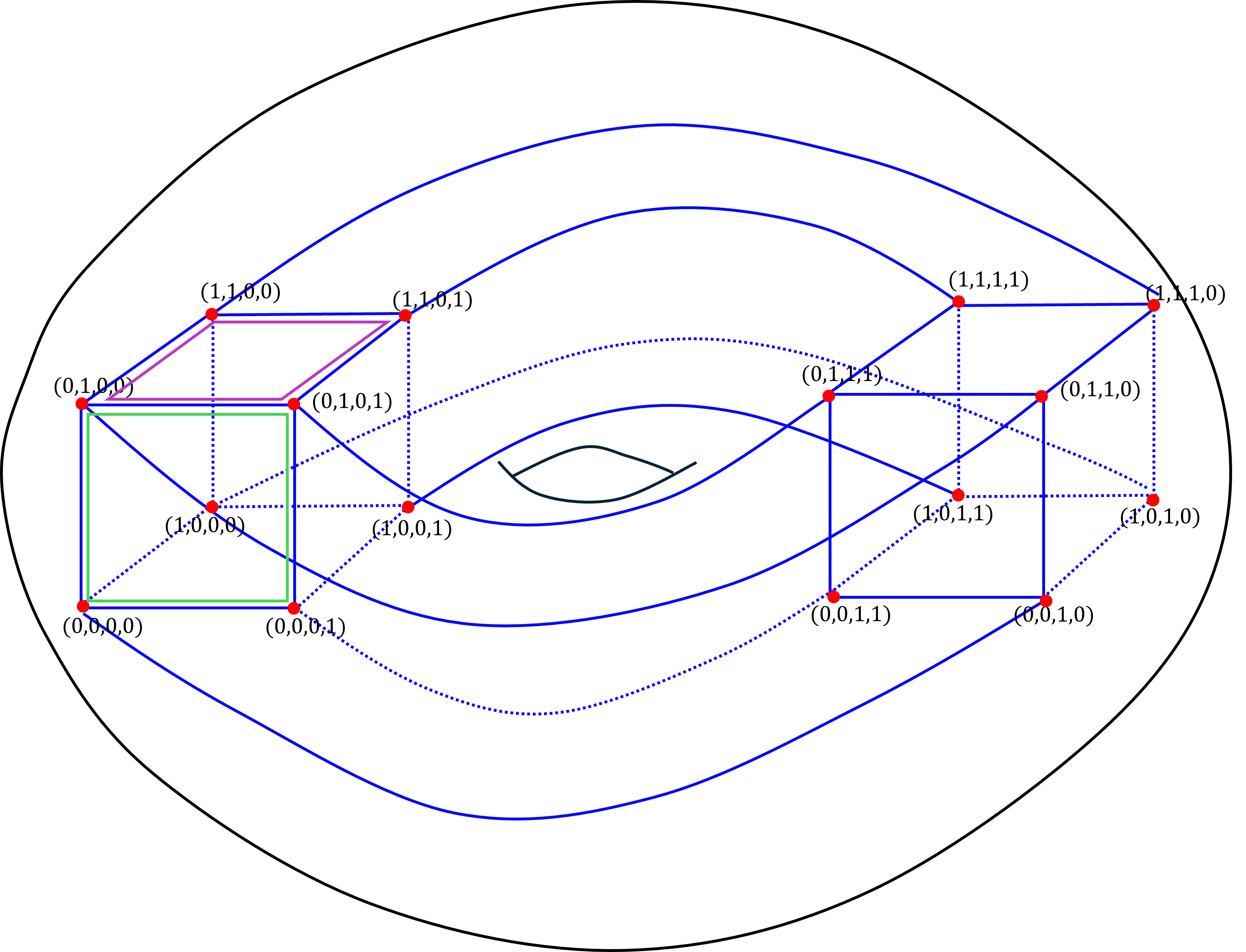}
\caption{The loop in green is the boundary of a square in the 4-cube that is mapped to a homotopically nontrivial loop on the torus, whereas the loop in purple is the boundary of a square in the 4-cube that is mapped to a square in the torus.}
\label{whichsquares}
\end{figure}

\begin{figure}
\centering
\includegraphics[height=0.95\textheight]{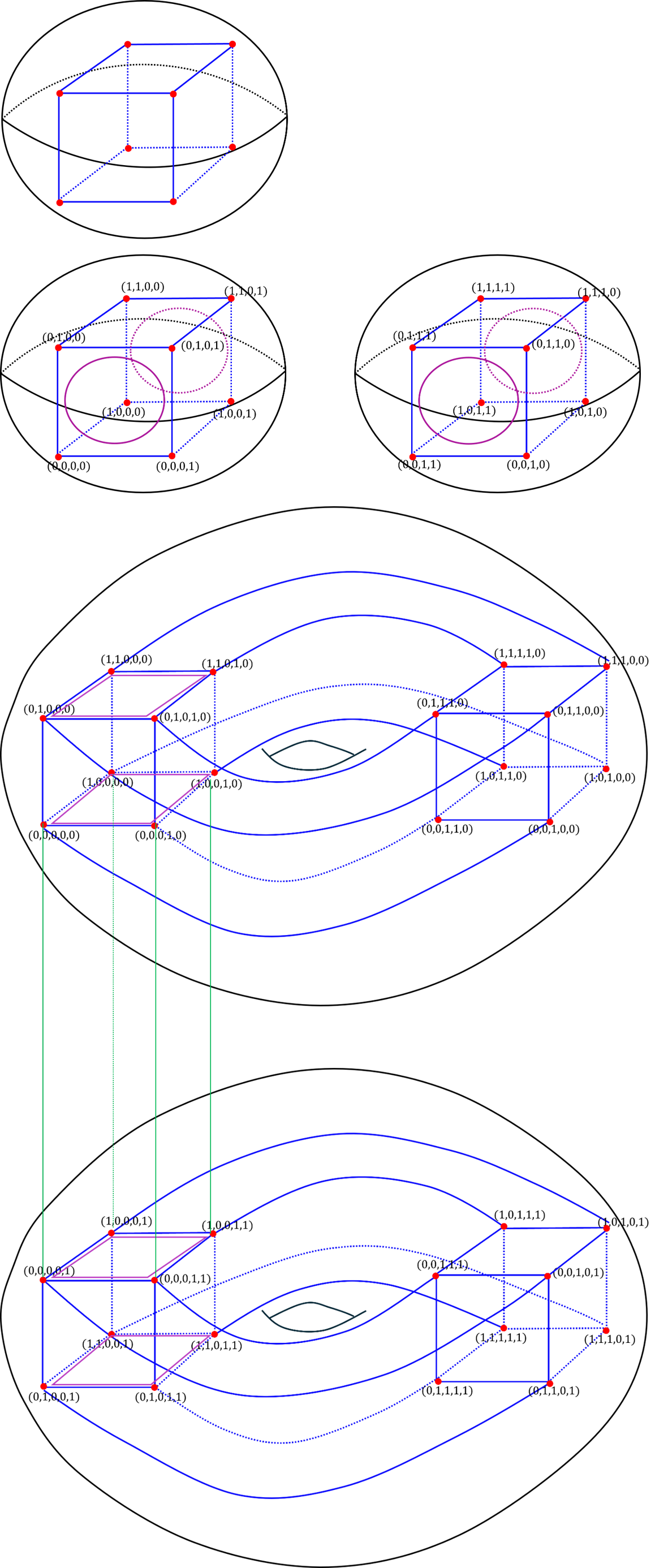}
\caption{}
\label{inductionfig}
\end{figure}

\begin{lem}
\label{squaresonsurface}
The ordering of the $m$-tuples that describe the coordinates of vertices can be chosen such that a square is a realised square iff it is of the form $\{(i,j), (a_{1}, \ldots, a_{m-2})\}$ with $j=i+1 \mod m$. 
\end{lem}

\begin{proof}
The proof is by induction. The lemma holds in the case $m=3$. This is a decomposition of the 2-sphere into 8 right-angled triangles, with dual graph the 1-skeleton of the 3-cube, as shown in Figure \ref{inductionfig}. The next simplest example has dual graph given by the 1-skeleton of the 4-cube. The 4-cube is obtained by taking the cartesian product of the 3-cube with an interval; the graph $C^{4}$ is obtained by taking two copies of $C^{3}$ and adding an edge between each pair of corresponding vertices. Take two copies of the decomposition of the 2-sphere by right-angled triangles. When the right-angled triangles in the decompositions of the 3-sphere are replaced by their dual realised squares, the extra edges give a pair of holes (shown in purple in Figure \ref{inductionfig}) on the 2-spheres. A 2-sphere with 2 disks cut out is a cylinder. Gluing the two cylinders together gives a torus, with a decomposition into squares with dual graph $C^{4}$. After performing this operation, the statement of the lemma still holds.

To obtain the next simplest example, take two copies of the torus with decomposition given by squares with dual graph $C^{4}$. Adding an extra edge to each of the squares of the decompositions to obtain right-angled pentagons creates four holes in each copy of the torus as shown in purple in Figure \ref{inductionfig}. For each of these holes, glue in an annulus, with one boundary component glued to a hole in one of the holed tori, and the other boundary component glued to the corresponding hole in the second holed torus. One such annulus is shown in green in Figure \ref{inductionfig}. This gives a surface of genus 5, decomposed into right-angled pentagons, with dual graph to the decomposition given by $C^{6}$. Once again, this operation preserves the statement of Lemma \ref{squaresonsurface}. All of our family of examples can be obtained by iterating this construction, each time preserving the statement of the lemma.
\end{proof}

From now on, the geodesics in the set $C$ from Theorem \ref{systolesandtesselation} will be called systoles, and an ordering will be assumed for which the statement in Lemma \ref{squaresonsurface} holds. 

\begin{rem}
\label{twosquares}
Lemma \ref{squaresonsurface} ensures that when the graph $C^{m}$ is mapped to the surface, each vertex is on the boundary of $m$ realised squares, and every edge is on the boundary of two realised squares.
\end{rem}

Using Remark \ref{twosquares}, every systole can be represented by four realised squares. The idea is best represented by an example.

\begin{ex}[A systole and the four squares it passes through]
\label{foursquares}
Consider the tessellated surface with dual graph $C^{5}$. The edge $e$ with endpoints given by $\{(0,0,0,0,0), (1,0,0,0,0)\}$ is on the boundary of the two realised squares $(\star,\star,0,0,0)$ and $(\star,0,0,0,\star)$. In $(\star,\star,0,0,0)$, the edge $e_{1}$ opposite $e$ has endpoints given by the vertices $\{(0,1,0,0,0,), (1,1,0,0,0)\}$ (the second entries of $\{(0,0,0,0,0), (1,0,0,0,0)\}$ are changed by 1). Similarly, in the square $(\star,0,0,0,\star)$, the edge $e_{2}$ opposite $e$ has endpoints given by $\{(0,0,0,0,1), (1,0,0,0,1)\}$.
The edge $e_{1}$ is contained in the boundary of the realised squares $(\star,\star,0,0,0)$ and $(\star,1,0,0,\star)$ and the edge $e_{2}$ in $(\star,\star, 0,0,1)$ and $(\star,0,0,0,\star)$. The four faces
\begin{equation*}
(\star,\star,0,0,0),\ \\
(\star,1,0,0,\star),\ \\
(\star,\star,0,0,1),\ \\
(\star,0,0,0,\star)
\end{equation*}
represent a systole. In all four squares, note that the third and fourth entries are the same (both 0), and the first coordinate is always $\star$. Consequently, the second star is always in position 2 or 5. These four squares are all contained in a 3-dimensional face of the 5-dimensional cube. This systole will be denoted by $\{1, (0,0)\}$. In the tessellated surface with dual graph $C^{5}$, it follows from Remark \ref{twosquares} that a systole can be uniquely determined by $\{i, (a_{1}, a_{2})\}$, where $i\in \{1, 2,\ldots, 5\}$ and $a_{1}, a_{2}\in \mathbb{Z}_{2}$.
\end{ex}

Generalising the construction from Example \ref{foursquares}, a systole will be written as $\{i, (a_{1}, \ldots, a_{k})\}$, where $i\in \{1, 2, \ldots, m\}$, $k=m-3$ and $a_{1}, a_{2}, \ldots, a_{k}\in \mathbb{Z}_{2}$. For the tessellated surface with dual graph $C^{m}$, this gives $2^{m-3}m$ systoles. Also generalising  the construction from Example \ref{foursquares}, one sees that every systole passes through four realised squares, and hence intersects four other systoles.


By Lemma \ref{squaresonsurface}, a necessary condition for the systoles $\{i, (a_{1}, \ldots, a_{k})\}$ and $\{j, (b_{1}, \ldots, b_{k})\}$ to intersect is that $i-j=\pm 1 \mod m$. If $j=i+1 \mod m$ then another necessary condition is that $(a_{2}, \ldots, a_{k})=(b_{1}, \ldots, b_{k-1})$. These two conditions are sufficient.

\textbf{Genus 5 example.} Recall that the example with $m=5$ gives a genus 5 surface. Our calculations show that the smallest minimal filling set contained in the set $C$ of systoles at the critical point has cardinality 8. One such minimal filling set is shown in Figure \ref{8systoles} and a fundamental domain is given in Figure \ref{fundaDomain}.


\begin{figure}[!htp]
\centering
\includegraphics[width=0.8\textwidth]{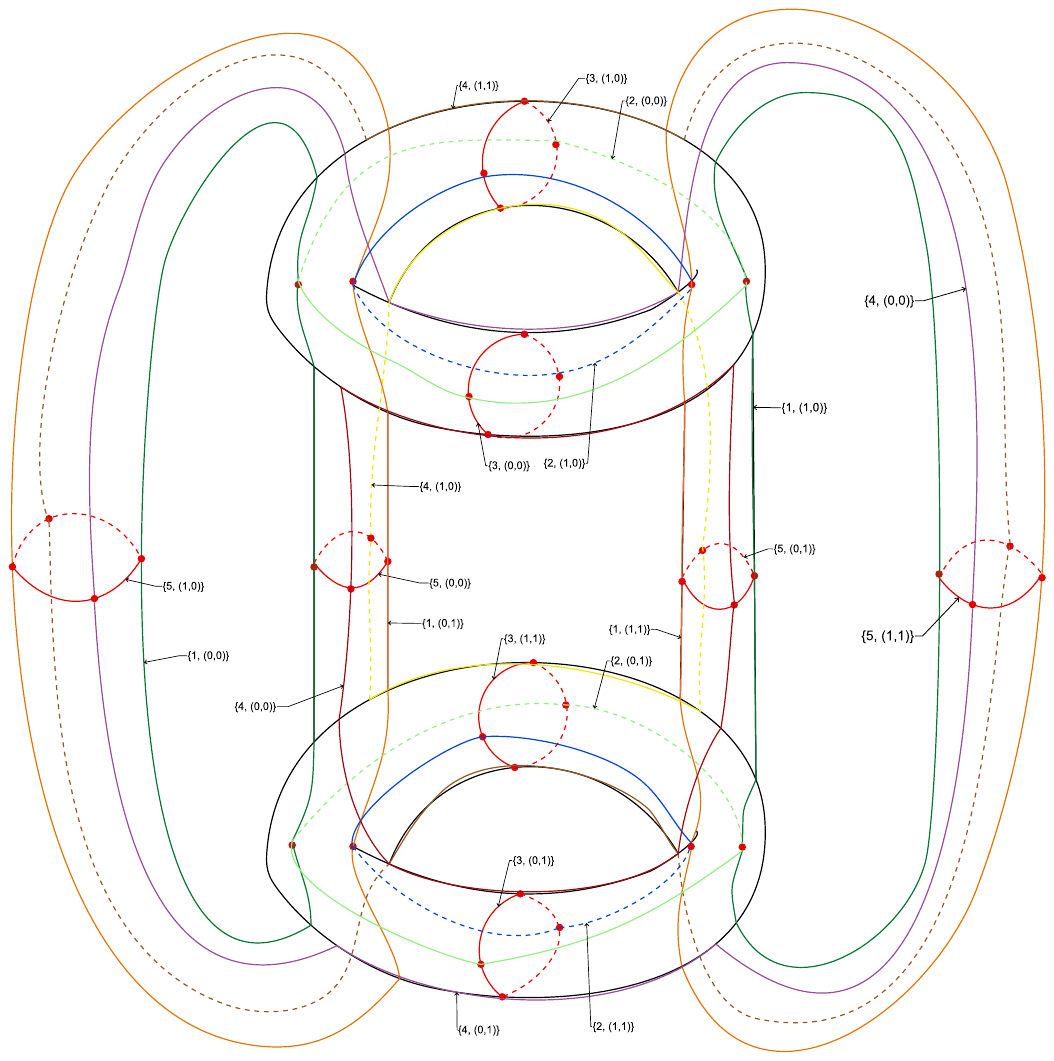}
\caption{All systoles in the genus 5 example.}
\label{Allsystoles}
\end{figure}

\begin{figure}[!htp]
\centering
\includegraphics[width=0.8\textwidth]{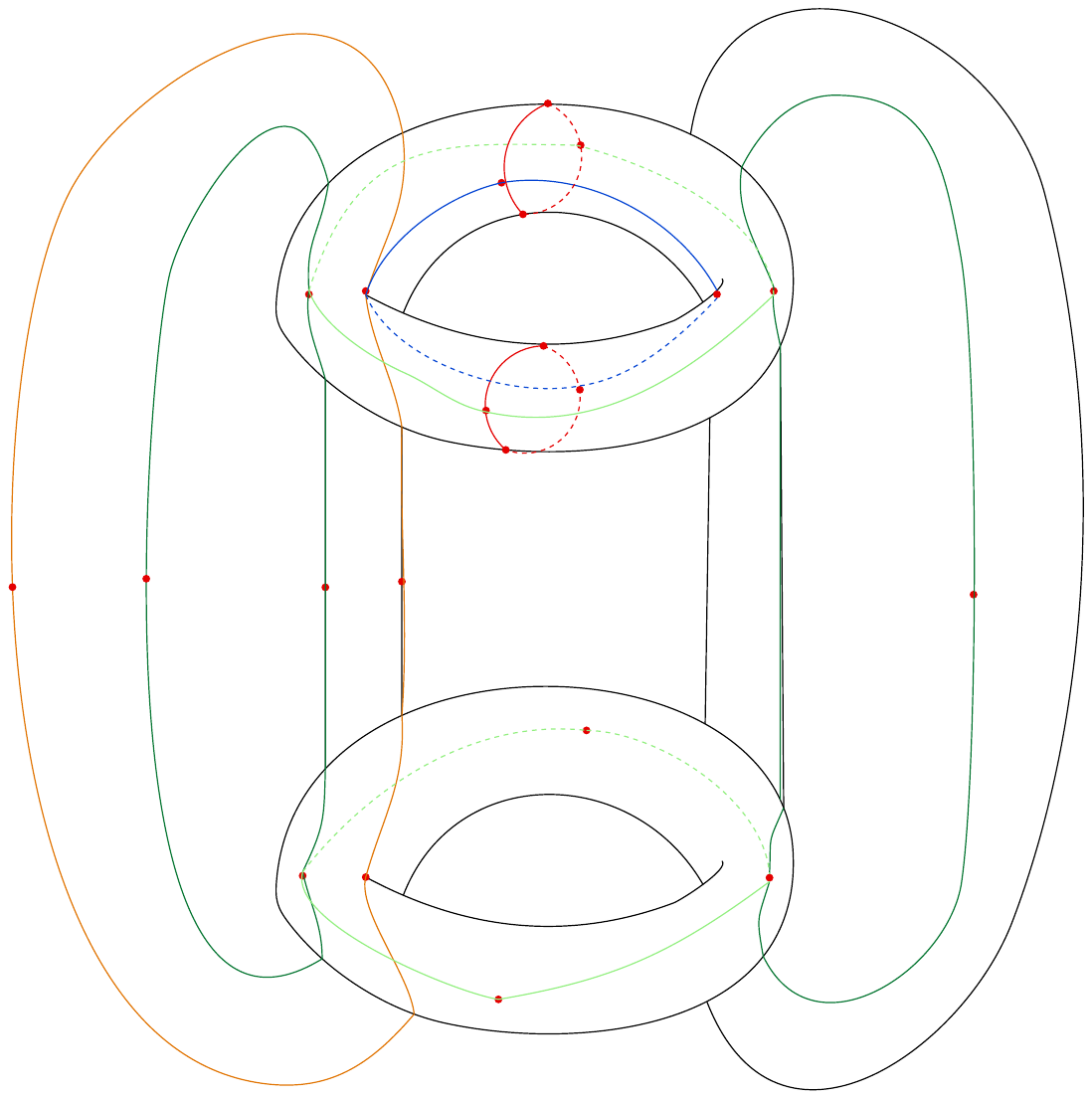}
\caption{A minimal filling set of 8 geodesics in the genus 5 surface.}
\label{8systoles}
\end{figure}

\textbf{Calculating the automorphism group.} The automorphism group of a tessellated surfaces in our family of examples with dual graph $C^{m}$ are isomorphic to the subgroup of the automorphism group of the graph $C^{m}$ that maps realisable squares to realisable squares. The following lemma will be assumed from now on without comment.

\begin{lem}
For every surface in our family of examples, the automorphism group acts transitively on the set of systoles.
\end{lem}
\begin{proof}
A systole is uniquely determined by an edge of the tessellation. Each such edge corresponds to an edge of the dual graph $C^{m}$. It therefore suffices to prove that the automorphism group acts transitively on the edges of $C^{m}$. 

First note that the automorphism of $C^{m}$ that changes one entry in the coordinate description of the vertices of $C^{m}$ leaves the set of realised squares invariant by Lemma \ref{squaresonsurface}. This corresponds to a reflection of $C^{m}$ through an $(m-1)$-dimensional cube parallel to a pair of opposite faces of $C^{m}$. Moreover, the automorphism that fixes a vertex $v$ of $C^{m}$ and permutes the edges emanating from $v$ in a way that respects the cyclic ordering of the edges also leaves the set of realised squares invariant by Lemma \ref{squaresonsurface}. This set of automorphisms act transitively on the edges of $C^{m}$.
\end{proof}

\section{Index Calculations}
\label{secindex}
This section describes the algorithm used for computing minimal filling sets and indexes of the critical points described in Section \ref{secexamples}.
\subsection{Finding Minimal Filling Sets}
\label{subminimal}

The integer linear program for finding minimal filling sets will now be explained.

Let $C=\{ c_{1}, \ldots, c_{n}\}$ be the set of all systoles, where $n=2^{m-3}m$.
Say that two systoles \textit{intersect} iff they pass through a common realised square.
For $i,j \in \{ 1, \ldots, n \}$, define $a_{i,j}$ as follows:
\[
a_{i,j}=
\begin{cases}
0 &\text{if }i=j, \\
1 & \text{if }i\neq j \text{ and }c_{i} \text{ and }c_{j}\text{ intersect},\\
0 & \text{if } i\neq j\text{ and } c_{i} \text{ and } c_{j}\text{ are disjoint}.
\end{cases}
\]

Recall that Lemma \ref{squaresonsurface} gave necessary and sufficient conditions for the systoles $\{i, (a_{1}, \ldots,  \allowbreak a_{k})\}$ and $\{j, (b_{1}, \ldots, b_{k})\}$ to intersect. Suppose $j=i+1 \mod m$, then the systoles intersect iff $(a_{2}, \ldots, a_{k})=(b_{1}, \ldots, b_{k-1})$.

To illustrate the general idea, let us first formulate the problem of finding minimal filling sets
as an integer quadratic program (IQP):
\begin{align*}
\text{Minimize} \quad & \sum_{i=1}^n x_i \\
\text{subject to} \quad & \sum_{i=1}^n \sum_{j=1}^i a_{i,j}x_ix_j \geq M(m), \\
& \sum_{j=1}^n a_{i,j}x_j \geq 1 \quad \forall i \in \{1,\ldots,n\},
\end{align*}
where $x_1, \ldots, x_n \in \{ 0, 1 \}$ and $M(m) = 1 - 2^{m-2}(4 - m)$ is a lower bound on the number of intersections derived from Euler characteristic arguments. The first constraint ensures a sufficient number of intersections, while the second constraint guarantees that every systole in the set intersects at least one other systole in $C$. Our solution $D$ corresponds to $D = \{ c_i\in C\ |\  x_i = 1 \}$.

Our goal is to replace the quadratic constraints by linear constraints.
For this aim we make use of the fact that systole pairs intersect at most once. In total there are $2n$ realised squares $\{ s_1, \ldots, s_{2n} \}$. We define the $2n\times n$ matrix $N$ with entries $n_{i,j}$ given by:
$$
n_{i,j} = 
\begin{cases} 
1 & \text{if } s_i \in c_j, \\
0 & \text{otherwise}.
\end{cases}
$$

Here $s_{i}\in c_{j}$ if $c_{i}$ is represented by a set of realisable squares containing $s_{i}$.

Introduce binary variables $y_1, y_2, \ldots, y_{2n}$, where $y_i=1$ if $s_i$ represents an intersection of $c_j $ and $c_k$
for some $j,k \in \{ 1, \ldots, n \}$ with $x_j = x_k = 1$ and $y_i = 0$ otherwise.
For convenience, we also introduce real variables $t_1, \ldots, t_n$ which count
the number of intersections for each systole in our solution $D$.
The above quadratic integer program is transformed into the following integer linear program (ILP):
\begin{equation}  
\label{linerPro2}
\begin{aligned}[b]  
\text{Minimize} \quad & \sum_{j=1}^n x_j  \\
\text{subject to} \quad & \sum_{i=1}^{2n} y_i \geq M(m),   \\
& t_j = \sum_{i=1}^{2n} n_{i,j}y_i \quad \forall j \in \{1,\ldots,n\}, \\
& x_j \leq t_j \leq 4x_j,   \\
& \sum_{j=1}^n a_{i,j}x_j \geq 1 \quad \forall i \in \{1,\ldots,n\}.
\end{aligned}
\end{equation}
Note that $t_j=0$ iff $x_j=0$ and $t_j \geq 1$ iff $x_j=1$.
Here $t_j \geq 1$ means that the $j$th systole is in the set.

The set of constraints given in the ILP above does not guarantee that the solutions are filling. For $m=5$ and $m=6$ this is the case. For $m = 5$, there is a minimal filling set of size $8$ unique up to isomorphism. For $ m = 6 $, there are  $6$ minimal filling sets of size 24 up to isomorphism.

In the $ m = 7 $ case, a lower bound of $62$ is found using ILP \eqref{linerPro2}. None of the solutions found were filling; this was a problem for all larger values of $m$ with which we worked. This made it necessary to add extra constraints to the ILP in \eqref{linerPro2} in order to obtain filling sets. Informally these constraints ensure the systoles intersect the boundary of the subsurfaces of $S_{g}$ filled by the solution set. The construction of these curves, denoted $C_{\mathrm{fill}}$, will now be explained.

\begin{figure}[!htp]
\centering
\includegraphics[width=0.8\textwidth]{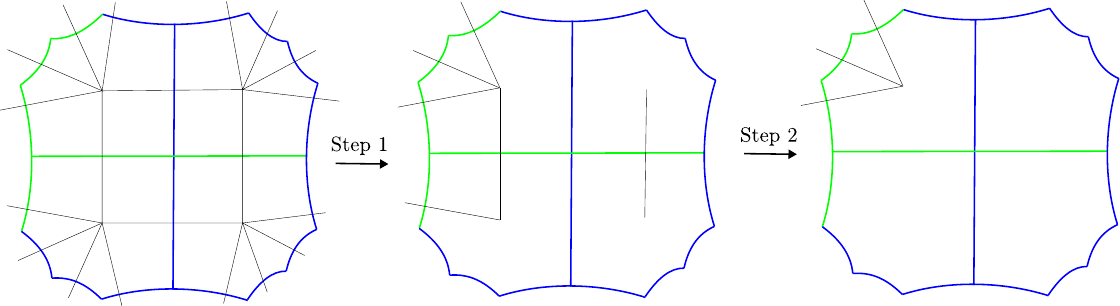}
\caption{A toy example illustrating the steps of the algorithm for constructing the simplified graph $\widetilde{C}^m$. Segments of systoles in the set $s$ are shown in blue, and segments of systoles not in the set are green. A portion of the graph $C^{m}$ is shown in black on the left. Note that any curve that does not intersect any of the blue curves must be in the subsurface containing $\tilde{C}^{m}$, a portion of which is shown in the figure on the right.}
\label{toyexample}
\end{figure}

Given a non-filling solution $s$ to the ILP, define the \textit{simplified graph} $\widetilde{C}^m$ to be the graph obtained as follows: 
\begin{enumerate}[label={Step \arabic*)}, leftmargin=*, align=left]
    \item Starting with $C^{m}$, delete every edge that intersects any curve in $s$.
    
    \item For every connected component of the graph from Step 1:
    \begin{itemize}[leftmargin=2em, label=\textbullet]
        \item Remove all vertices with valency $1$.
        \item If it exists, delete a vertex with valency $2$, that is,
        a vertex of a realizable square of the graph from the previous step.
    \end{itemize}
    
    \item Iterate Step 2 until this is no longer possible. If there is a vertex $v$ of valency greater than 3, and such that an edge $e$ incident on $v$ is on the boundary of only one realisable square in the graph, delete $e$.
\end{enumerate}

An edge $e$ in $\widetilde{C}^m$ is called a \textit{boundary edge} if there exists a realizable square in $C^m$ containing this edge, and the simplified graph $\widetilde{C}^m$ only contains 1 or 2 edges of this realisable square. The set $C_{\mathrm{fill}}$ consists of the generators of the fundamental group of the subgraph of $\widetilde{C}^m$ containing only boundary edges. The purpose of step 3 in the algorithm above is to ensure that these generators are all noncontractible in $S_{g}$.


\textbf{Symmetry breaking.} A combinatorial technique known as symmetry breaking will be used to restrict the ILP to a considerably smaller but suitably representative set.

For $C^7$, we use $\{i,(a_1,a_2,a_3,a_4)\}$ for $i=1,\ldots,7$, $a_k\in \{0,1\}$ to represent systoles. Let $C_i=\{i,(a_1,a_2,a_3,a_4)\ |\  a_k\in \{0,1\} \}$, $i=1,\ldots,7$, then $|C_i|=16$. Let $G$ be the automorphism group of our tessellated surface. Note that $G$ is transitive on $\{C_1,\ldots C_7\}$.

Let $C=\{ c_1,\ldots,c_{112}\}$ be the set of systoles for $C^7$, $x_i \in\{0,1\}$ where $x_{i}=1$ implies $c_{i}$ is in the solution set. We have computed that the smallest cardinality of minimal filling set is 62, 63 or 64.

Suppose the solution vector $\mathbf{x} = \{x_1, x_1, \ldots, x_{112}\}$ corresponds to a systole configuration that forms a minimal filling set with the smallest possible cardinality. Since there are 112 systoles in total, and the smallest cardinality of a minimal filling set is 62, 63 or 64, if $\mathbf{x} = \{x_1, x_1, \ldots, x_{112}\}$ corresponds to a minimal filling set with the smallest possible cardinality, then $|\{i\ |\  i=1,\ldots,112\text{ and }x_{i}=1 \}|=62, 63,$ or 64. This implies $|\{i\ |\ i=1,\ldots,112 \text{ and }x_{i}=0\}|=48,49, $ or 50. For the partition $C_1, C_2, \ldots, C_7$, there must exist some $C_i$ such that $|\{k \ |\  c_{k} \in C_i \text{ and } x_{k}=0\}| \leq 7$. By transitivity, we may assume $C_1$ satisfies this condition.

Denote by $S$ the set of all possible combinations of at most 7 out of 16 elements:
$$
|S| = \sum_{k=0}^7 \binom{16}{k}.
$$
Use $v = \{i_1, i_2, \ldots, i_{16}\}$ to represent an element in $S$, where $i_k \in \mathbb{Z}_{2}$. Each combination $v = \{i_1, i_2, \ldots, i_{16}\} \in S$ corresponds to an initial assignment method. Under the action of the automorphism group, $S$ decomposes into 923 distinct orbits $\mathcal{O}_1, \mathcal{O}_2, \ldots, \mathcal{O}_{923}$. For each orbit $\mathcal{O}_i$, select a representative $v_i \in \mathcal{O}_i$ and add the following constraints to the ILP \eqref{linerPro2} to obtain ILP \eqref{linerPro}:
\begin{equation}  
\label{linerPro}
\begin{aligned}[b]  
& x_j= i_j, \text{ for } j\in \{1,2,\ldots 16 \},\\
&\sum_{j=1}^n x_j=N.
\end{aligned}
\end{equation}

The two conditions guarantee that every solution set intersects the orbit of
$C_{\mathrm{fill}}$ under the stabiliser subgroup of $C_{1}$.
These ILPs were all shown to be infeasible\footnote{We used Gurobi 11.0.3 on a Intel Xeon Platinum 8268 with 50 threads. The calculations for all 955 ILPs took around 3 CPU weeks.} for $N=62, 63$, and all solutions found for $N=64$ are filling. We concluded that the smallest cardinality of the minimal filling set is 64.

\subsection{Calculating the Index}
\label{subsecindex}

This subsection explains the algorithm for calculating the index of critical points.

A set of closed geodesics $C$ will be said to be \textit{eutactic} at a point $x\in \mathcal{T}_{g}$ if every derivation $v$ in $T_{x}\mathcal{T}_{g}$ has the property that either $vL(c)(x)=0$ for all $c\in C$ or there exists a $c_{1}\in C$ for which $vL(c_{1})(x)>0$ and a $c_{2}\in C$ for which $vL(c_{2})(x)<0$.

In \cite{Akrout}, it was shown that $x\in \mathcal{T}_{g}$ is a critical point iff there is some set $C$ of geodesics such that $C$ is the set of systoles at $x$ and $C$ is eutatic at $x$. At a critical point $x$, it follows from local finiteness and convexity of length functions along Weil-Petersson geodesics \cite{WolpertHessian} that the orthogonal complement of $\{\nabla L(c)(x)\ |\ c\in C\}$ is a subspace of $T_{x}\mathcal{T}_{g}$ on which $f_{\mathrm{sys}}$ is increasing to second order. From eutacticity and local finiteness it follows that the span of $\{\nabla L(c)(x)\ |\ c\in C\}$ consists of the directions in which $f_{\mathrm{sys}}$ is decreasing at $x$. The index of the critical point is therefore the dimension of the span of $\{\nabla L(c)(x)\ |\ c\in C\}$.

As critical points represent hyperbolic surfaces that often have a larger symmetry group than the hyperbolic surfaces represented by points in $\mathcal{T}_{g}$ near the critical point, the dimension of the span of the gradients at nearby points does not give a good estimate of the dimension of the span at the critical point.

\begin{figure}
\centering
\includegraphics[width=0.8\textwidth]{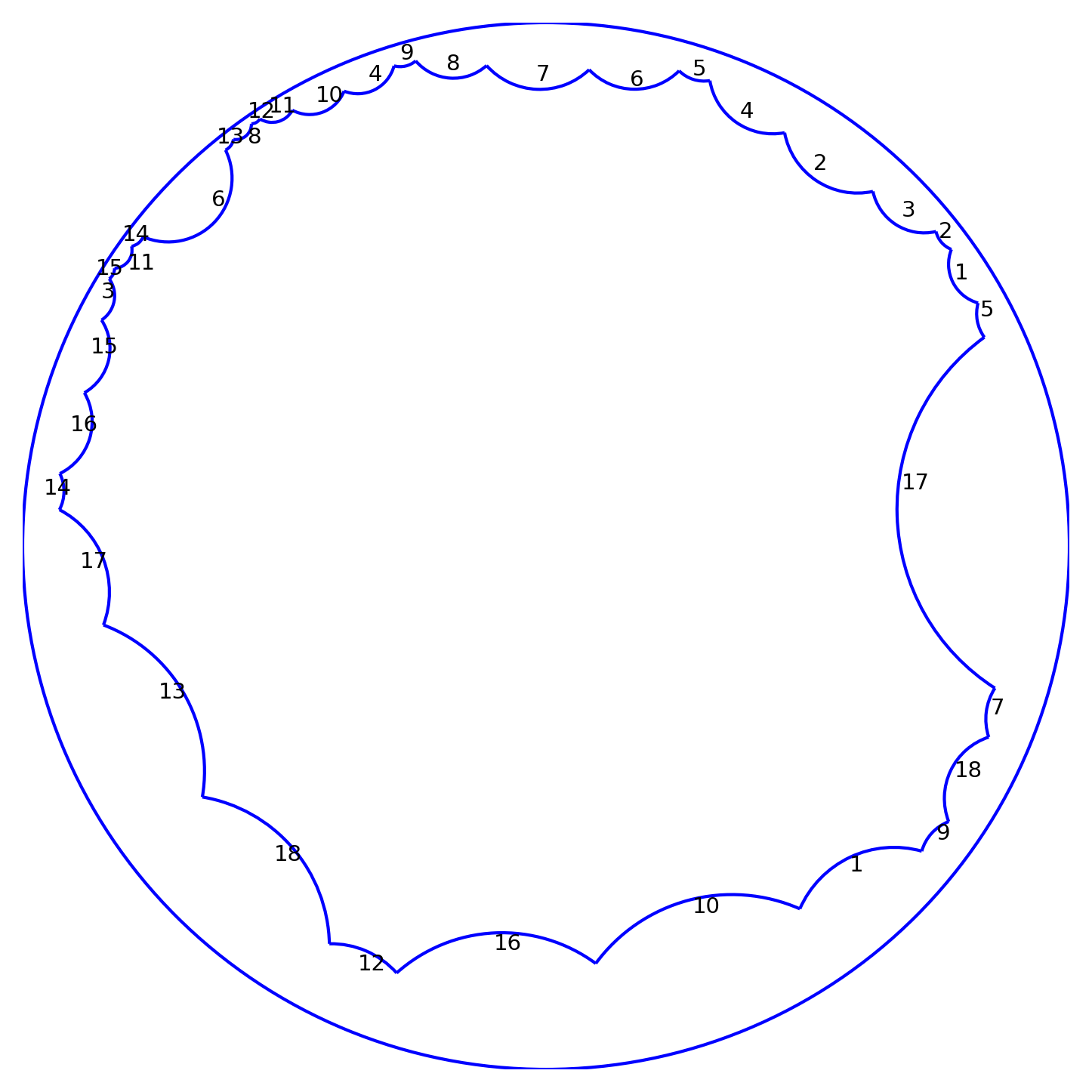}
\caption{A fundamental domain for $C^5$ in the disk model. Edges that correspond to the same number are glued together in the surface.}
\label{fundaDomain}
\end{figure}

In all our examples, a minimal filling set $C'\subset C$ determines a fundamental domain $\mathcal{F}(C')$ consisting of a polygon obtained by cutting $S_{g}$ along the geodesics in $C'$. Filling sets of curves with this property were studied in \cite{minimalfilling}. Figure \ref{fundaDomain} shows such a fundamental domain for the genus 5 example, obtained by cutting $S_{5}$ along the systoles shown in Figure \ref{8systoles}.  To obtain the hyperbolic surface from the fundamental domain, pairs of edges are glued together to form a single edge in the surface. Figure \ref{fundaDomain} shows the pairing of the edges of the fundamental domain, which should not be confused with the counter clockwise labels of (single) edges.

\begin{figure}[ht]
\centering
\includegraphics[width=\textwidth]{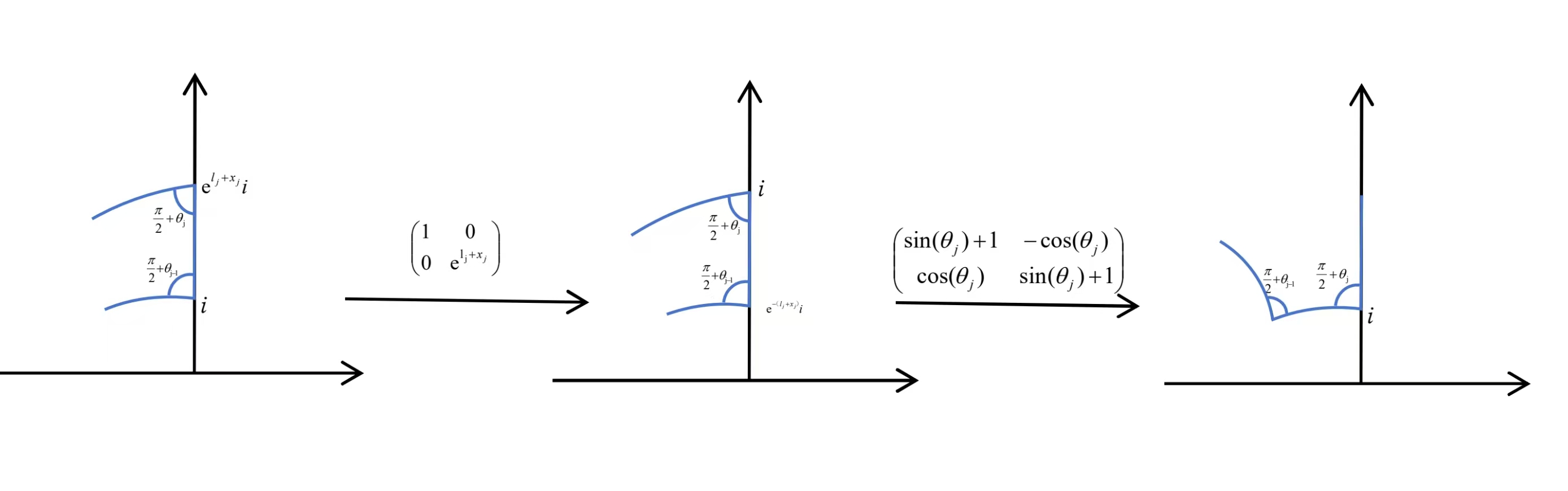}
\caption{$M_j$ is a composition of scaling and rotation.}
\label{M}
\end{figure}
To calculate the index of critical point $x$ in our family of examples, a set of coordinates for $\mathcal{T}_{g}$ on a neighbourhood of $x$ was chosen, consisting of lengths of edges of $\mathcal{F}(C')$ and angles at the vertices of $\mathcal{F}(C')$. These parameters are subject to a number of constraints. Firstly, edges of $\mathcal{F}(C')$ that are glued together have the same length. To derive the second set of constraints, let $n$ be the number of edges of $\mathcal{F}(C')$. Embed $\mathcal{F}(C')$ in the upper half plane model in such a way that the endpoint of edge 1 is at the complex number $i$ and the first edge is above $i$ on the imaginary axis. The edges are assumed to be labelled in such a way that edge $j+1$ follows edge $j$ as the edges are traversed counter clockwise. The angle $\frac{\pi}{2}+\theta_{j}$ is the angle at the vertex of $\mathcal{F}(C')$ between edge $j$ and $j+1$. The edge $j$ determines the matrix $M_{j}\in PSL(2,\mathbb{R})$ obtained by scaling the following

\begin{equation}
\begin{pmatrix}
\sin(\theta_{j})+1 & -e^{l_{j}+x_{j}}\cos(\theta_{j})\\
\cos(\theta_{j})&e^{l_{j}+x_{j}}(\sin(\theta_{j})+1)
\end{pmatrix}
\end{equation}

The transformation $M_j$ is shown as figure \ref{M}. The matrix $M_j$ represents a translation through distance $l_{j}$ followed by a rotation through angle $\frac{\pi}{2}+\theta_{i}$. Here $l_{i}$ is the length of edge $i$ at the critical point $x$, $x_{i}$ is equal to the parameter $y_{j}$ that measures the change in length of the $j$th pair of edges, where the $j$th pair contains the $i$th edge. This gives the constraint
\begin{equation}
\label{polygoneq}
M_{4M(m)}\ldots M_{1}=I~in~PSL(2,\mathbb{R})
\end{equation}

These constraints make it possible to eliminate dependent angle and length parameters. For example, if our minimal filling set $ C' $ has exactly $ M(m) $ intersection points (as is the case for every example we computed), then $ \mathcal{F}(C') $ will contain $ 4M(m) $ edges, forming $ 2M(m) $ pairs of edges. From equation \eqref{polygoneq}, we can eliminate three length parameters, leaving $ 2M(m) - 3 $ length parameters and $ M(m) $ angle parameters. This results in a total of $ 3M(m) - 3 $ independent parameters, which exactly equals $ 6g - 6 $.

To calculate the length of a geodesic $c$ in $C\setminus C'$, the lengths of the arcs obtained as the intersection of $c$ with $\mathcal{F}(C')$ are calculated. This is done by using constrained optimisation to calculate the length of the geodesic intersecting the same edges of $\mathcal{F}(C')$ in the same order as $c$.

For any $ c \in C $, we can numerically compute $ \frac{\partial L(c)}{\partial y_i} $ and $ \frac{\partial L(c)}{\partial \theta_j} $ for some independent parameters $y_i$ and $\theta_j$.  The number of linearly independent differentials $\{dL(c)(x)\ |\ c\in C\}$ then gives the index of the critical point $x$.

\subsection{Questions and future work}
\label{future}
The main motivation for the computations in this paper is in the search for combinatorial criteria to decide if/where level sets of length functions on $\mathcal{T}_{g}$ intersect transversely. Having such criteria would be extremely useful because the major source of difficulties in studying the geometry and topology of $\mathcal{P}_{g}$ is that transversality might break down in unpredictable ways. Some initial progress was made in \cite{Realisable} in deciding what nonfilling sets of curves can be realised as systoles, but for general sets of curves this is a difficult problem without more combinatorial criteria for characterising where transversality breaks down.

As one can see from Table 1, the number of geodesics in the set of systoles in our family of critical points representing independent homology classes gives a lower bound on the index. This leads to the following question:

\begin{ques}
Suppose $C$ is a set of simple, nonseparating, oriented closed geodesics on $S_{g}$, representing independent homology classes in $H_{1}(S_{g}; \mathbb{Q})$. Away from minima of $C$ in $\mathcal{T}_{g}$, are the gradients $\{\nabla L(c)\ |\ c\in C\}$ linearly independent?
\end{ques}

Minima of $C$ are defined on page 14 of \cite{MorseSmale}. Informally, the question asks if the assumptions on $C$ are sufficient to ensure that $\{\nabla L(c)\ |\ c\in C\}$ are as linearly independent as possible. It might help to additionally assume that cutting $S_{g}$ along the geodesics in $C$ does not give any triangles. Ruling out triangular complementary regions is a simplifying assumption that ensures the geodesics do not undergo any type III Reidemeister moves along paths in $\mathcal{T}_{g}$. It is not understood what might happen to the dimension of the span of $\{\nabla L(c)\ |\ c\in C\}$ at such points.

More generally,
\begin{ques}
If $C$ is a set of geodesics that determine a regular tessellation of $S_{g}$ by $m$-gons with $m\geq 5$, are $\{\nabla L(c)\ |\ c\in C\}$ linearly independent on $\mathcal{T}_{g}$ away from minima of $C$?
\end{ques}

\printbibliography

\end{document}